\documentclass[reqno,a4paper]{amsart}
\usepackage{amsmath}
\usepackage{amssymb,amsfonts}
\usepackage{amsthm}
\usepackage{enumerate}
\usepackage{cite}
 
\usepackage[mathscr]{eucal}
\usepackage{xcolor}

\theoremstyle{plain}
\newtheorem{theorem}{Theorem}[section]

\newtheorem{lemma}{Lemma}[section]
\newtheorem{corol}{Corollary}[theorem]

\theoremstyle{definition}

\theoremstyle{remark}

\numberwithin{equation}{section}

\begin{document}

\title[A note on the ternary purely exp. Diop. eq. $A^x+B^y=C^z$ with $A+B=C^2$]
{A note on the ternary purely exponential Diophantine equation $A^x+B^y=C^z$ with $A+B=C^2$}

\author{El\.{I}f K{\i}z{\i}ldere, Maohua Le and G\"{o}khan Soydan}

\address{{\bf Elif K{\i}z{\i}ldere}\\
	Department of Mathematics \\
	Bursa Uluda\u{g} University\\
	16059 Bursa, Turkey}
\email{elfkzldre@gmail.com}

\address{{\bf Maohua Le}\\
	Institute of Mathematics, Lingnan Normal College\\
	Zhangjiang, Guangdong, 524048 China}

\email{lemaohua2008@163.com}

\address{{\bf G\"{o}khan Soydan} \\
	Department of Mathematics \\
	Bursa Uluda\u{g} University\\
	16059 Bursa, Turkey}
\email{gsoydan@uludag.edu.tr }

\newcommand{\acr}{\newline\indent}
\newcommand{\ideal}[1]{\langle #1\rangle}
\newcommand*{\Comb}[2]{{}^{#1}C_{#2}}%


\begin{abstract}
Let $\ell, m, r$ be fixed positive integers such that $2\nmid \ell$, $3\nmid \ell m$, $\ell>r$ and $3\mid r$. In this paper, using the BHV theorem on the existence of primitive divisors of Lehmer numbers, we prove that if $\min\{r\ell m^2-1,(\ell-r)\ell m^2+1\}>30$, then the equation $(r\ell m^2-1)^x+((\ell -r)\ell m^2+1)^y=(\ell m)^z$ has only the positive integer solution $(x,y,z)=(1,1,2)$.
\end{abstract}

\maketitle

\section{Introduction}
Let $\mathbb{Z}$, $\mathbb{N}$ be  the sets of all integers and positive integers, respectively. Let  $A, B, C$ be coprime positive integers with $\min\{A,B,C\}>1$. Recently, there are many papers have discussed the solutions $(x,y,z)$ of the ternary purely exponential Diophantine equation 
\begin{equation}\label{1.1}
A^x+B^y=C^z, \ \ x, y, z \in \mathbb{N}
\end{equation}
for some triples $(A,B,C)$ with $A+B=C^2$ (see \cite{Ber,FY,KMS,MT,P,SL,T,TH1,TH2,WWZ}). See also the survey paper \cite{SCDT}, for more details about the equation \eqref{1.1}.

For example, N. Terai and T. Hibino \cite{TH2} proved that if
\begin{equation}\label{1.2}
A=3pm^2-1,\, B=(p-3)pm^2+1,\, C=pm,
\end{equation}
where $p$ is an odd prime with $p<3784$ and $p\equiv 1\pmod{4}$, $m$ is a positive integer with $3\nmid m$ and $m\equiv 1\pmod{4}$, then \eqref{1.1} has only the solution $(x,y,z)=(1,1,2)$.

Let $\ell$, $m$, $r$ be positive integers such that
\begin{equation}\label{1.3}
2\nmid \ell, \ 3\nmid \ell m, \ \ell>r, \ 3\mid r.
\end{equation}
In this paper we consider \eqref{1.1} for the case
\begin{equation}\label{1.4}
A=r\ell m^2-1, \ B=(\ell-r)\ell m^2+1, \ C=\ell m.
\end{equation}
Using the BHV theorem on the existence of primitive divisors of Lehmer numbers due to Y. Bilu, G. Hanrot and P.M. Voutier \cite{BHV}, we prove a general result as follows.
 \begin{theorem}[Main theorem]\label{mainth}
Let $A$, $B$, $C$ satisfy \eqref{1.4} with \eqref{1.3}. If $\min\{r\ell m^2-1,(\ell -r)\ell m^2+1\}>30$, then \eqref{1.1} has only the solution $(x,y,z)=(1,1,2)$.
\end{theorem}
The above theorem can improve many existing results. For example, by comparing \eqref{1.2} and \eqref{1.4}, we can obtain the following corollary immediately.

\begin{corol}
Let $A$, $B$, $C$ satisfy \eqref{1.2}. If $p\ge 11$ and $3 \nmid m$, then \eqref{1.1} has only the solution $(x,y,z)=(1,1,2)$.
\end{corol} 

\section{Preliminaries} \label{Section 2}

Let $D_1, D_2, k$ be fixed positive integers such that $\min \{{D_1,D_2}\}>1$ and $\gcd (D_1,D_2)=\gcd (k,2D_1D_2)=1$. We now introduce some results on the equation
\begin{equation}\label{2.1}
D_1X^2+D_2Y^2=k^Z, \  X,Y,Z \in \mathbb{Z}, \  \gcd(X,Y)=1, \ Z>0
\end{equation}
due to M. H. Le \cite{Le}. For any fixed solution $(X,Y,Z)$ of \eqref{2.1}, there exists a unique positive integer $L$ such that
\begin{equation}\label{2.2}
L \equiv -\dfrac{D_1X}{Y} \pmod k, \ 0<L<k.
\end{equation}
The positive integer $L$ is called the characteristic number of the solution of \eqref{2.1}, and denoted by $\ideal {X,Y,Z}$. Let $(X_0,Y_0,Z_0)$ be a fixed solution of \eqref{2.1} and let $L_0=\ideal {X_0,Y_0,Z_0}$. Further, let $S(L_0)$ denote the set of all solutions $(X,Y,Z)$ with
\begin{equation}\label{2.3}
\ideal{X,Y,Z} \equiv \pm L_0 \pmod k.
\end{equation}

\begin{lemma}[\ \cite{Le}]\label{Lemma 2.1}
$S(L_0)$ has a unique solution $(X_1,Y_1,Z_1)$ such that $X_1>0,Y_1>0$ and $Z_1\leq Z$, where $Z$ runs through all solutions $(X,Y,Z)$ in $S(L_0)$. Moreover every solution $(X,Y,Z)$ in $S(L_0)$ can be expressed as
$$
Z=Z_1t, \ t\in \mathbb{N}, \ 2\nmid t,
$$
$$
X\sqrt{D_1}+Y\sqrt{-D_2}=\lambda_1(X_1\sqrt{D_1}+\lambda_2Y_2\sqrt{-D_2})^t, \ \lambda_1,\lambda_2 \in \{1,-1\}.
$$
\end{lemma}

Let $\alpha,\beta$ be algebraic integers. If $(\alpha+\beta)^2$, $\alpha\beta$ are nonzero coprime integers, and $\alpha/\beta$ is not a root of unity, then $(\alpha,\beta)$ is called a Lehmer pair. Let $E=(\alpha+\beta)^2$ and $G=\alpha\beta$. Then we have
\begin{equation}\label{2.4}
\alpha=\frac{1}{2}(\sqrt{E}+\lambda\sqrt{F}), \ \beta=\frac{1}{2}(\sqrt{E}-\lambda\sqrt{F}), \ \lambda\in \{1,-1\},
\end{equation}
where $F=E-4G$. Further, one defines the corresponding sequence of Lehmer numbers by
\begin{equation}\label{2.5}
L_n(\alpha,\beta)=\begin{cases}
\ \dfrac{\alpha^n-\beta^n}{\alpha^2-\beta^2} & \text{if $2\mid n$,}\\
\\
\ \dfrac{\alpha^n-\beta^n}{\alpha-\beta}& \text{if $2\nmid n$,}\\
\end{cases}
n=0,1,2,\cdots.
\end{equation}
Obviously, $L_n(\alpha,\beta)$ $(n>0)$ are nonzero integers.

A prime $q$ is called a primitive divisor of the Lehmer number $L_n(\alpha,\beta)$ $(n>1)$ if $q\mid L_n(\alpha,\beta)$ and $q\nmid FL_1(\alpha,\beta)\dots L_n(\alpha,\beta)$.
\begin{lemma}[\ \cite{BHV}] \label{Lemma 2.2}
If $n>30$, then $L_n(\alpha,\beta)$ has primitive divisors.
\end{lemma} 

\section{Proof of Theorem}
  
In this section, we assume that $\ell,m,r$ are positive integers satisfying \eqref{1.3}. Let $d=\gcd (r\cdot \ell m^2-1,(\ell-r)\cdot \ell m^2+1)$. Since
\begin{equation}\label{3.1}
(r\ell m^2-1)+((\ell-r)\ell m^2+1)=(\ell m)^2,
\end{equation}
we have $d\mid \ell^2m^2$. Further, since $\gcd (\ell m, r\ell m^2-1)=1$, we get $d=1$ and
\begin{equation}\label{3.2}
\gcd (r\ell m^2-1, (\ell-r)\ell m^2+1)=1.
\end{equation}

\begin{lemma}\label{Lemma 3.1}
All solutions $(x,y,z)$ of the equation
\begin{equation}\label{3.3}
(r\ell m^2-1)^x+((\ell-r)\ell m^2+1)^y=(\ell m)^z, \ x,y,z \in \mathbb{N}
\end{equation}
satisfy $2\nmid xy$.
\end{lemma}

\begin{proof}
Since $\ell >r\geq 3$ by \eqref{1.3}, we see from \eqref{3.3} that $0\equiv (\ell m)^z \equiv (r\ell m^2-1)^x+((\ell-r)\ell m^2+1)^y \equiv (-1)^x+1 \pmod \ell$ and 
\begin{equation}\label{3.4}
2\nmid x.
\end{equation}

Since $3\mid r$ and $3\nmid \ell m$, by \eqref{3.3} and \eqref{3.4}, we have
$$
0 \not\equiv (\ell m)^z \equiv (r\ell m^2-1)^x+((\ell-r)\ell m^2+1)^y \equiv (-1)^x+(\ell^2m^2+1)^y
\equiv -1+2^y \\
$$
\begin{equation}\label{3.5}
\equiv -1+(-1)^y \equiv \begin{cases}
\ 0 \pmod 3 & \text{if $2\mid y$,}\\
\\
\ 1 \pmod 3 & \text{if $2\nmid y$.}\\
\end{cases}
\end{equation}

Hence, by \eqref{3.5}, we get $2\nmid y$. Thus, the lemma is proved.
\end{proof}

\begin{lemma}\label{Lemma 3.2}
If $2 \mid m$, then \eqref{3.3} has no solutions $(x,y,z)\neq(1,1,2)$.
\end{lemma}

\begin{proof}
Let $(x,y,z)$ be a solution of \eqref{3.3} with $(x,y,z)\neq(1,1,2)$. Since $(x,y)\neq(1,1)$, we have $\max \{x,y\}\geq2$, and by \eqref{3.1}, we get $z\geq3$. Hence, by Lemma \ref{Lemma 3.1}, we see from \eqref{3.3} that
\begin{equation}\label{3.6}
0 \equiv (\ell m)^z \equiv (r\ell m^2-1)^x+((\ell-r)\ell m^2+1)^y \equiv (r\ell m^2x-1)+((\ell-r)\ell m^2y+1) \equiv \ell m^2(rx+(\ell-r)y) \pmod {m^3},
\end{equation}
whence we get
\begin{equation}\label{3.7}
\ell (rx+(\ell-r)y) \equiv 0 \pmod m.
\end{equation}
If $2\mid m$, since $2\nmid xy$, then from \eqref{3.7} we obtain $\ell^2 y \equiv \ell^2 \equiv 0 \pmod 2$. But, since $2\nmid \ell$ by \eqref{1.3}, this is impossible. Thus, the lemma is proved.
\end{proof}

\begin{proof}[Proof of the Theorem]
We now assume that $(x,y,z)$ is a solution of \eqref{3.3} with $(x,y,z) \neq (1,1,2)$. By Lemma \ref{Lemma 3.2}, we have $2\nmid m$. Let
\begin{equation}\label{3.8}
P=r\ell m^2-1, \ Q=(\ell-r)\ell m^2+1, \ K=\ell m. 
\end{equation}
By \eqref{3.1} and \eqref{3.8}, we have $\min \{P,Q\}>1$ and $\gcd (P,Q)=\gcd (K,2PQ)=1$. Therefore, the results of \eqref{2.1} given in Section \ref{Section 2} can be applied to the equation
\begin{equation}\label{3.9}
PX^2+QY^2=K^Z, \ X,Y,Z \in \mathbb{Z}, \ \gcd (X,Y)=1, \ Z>0.
\end{equation}

We see from \eqref{3.1} and \eqref{3.8} that \eqref{3.9} has a solution $(X_0, Y_0, Z_0)=(1,1,2)$. Let $L_0=\ideal{1,1,2}$. Since
\begin{equation}\label{3.10}
P \equiv -1 \pmod K, \ Q \equiv 1 \pmod K,
\end{equation}
by \eqref{2.2} and \eqref{3.10}, we have $L_0 \equiv -P \equiv 1 \pmod K$ and
\begin{equation}\label{3.11}
L_0=1.
\end{equation}
Further, since $P+Q=K^2$ by \eqref{3.1}, all solutions $(X,Y,Z)$ of \eqref{3.9} satisfy $Z\geq2$. This implies that $(1,1,2)$ is the least solution of $S(1)$, where $S(1)$ is the set of all solutions $(X,Y,Z)$ of \eqref{3.9} with
\begin{equation}\label{3.12}
\ideal{X,Y,Z} \equiv \pm 1 \pmod K.
\end{equation}

Since $2\nmid xy$, we see from \eqref{3.3} and \eqref{3.8} that \eqref{3.9} has a solution
\begin{equation}\label{3.13}
(X,Y,Z)=(P^{(x-1)/2},Q^{(y-1)/2},z).
\end{equation}
Let $L=\ideal{X,Y,Z}$. By \eqref{2.2}, \eqref{3.3}, \eqref{3.8} and \eqref{3.10}, we have
\begin{equation}\label{3.14}
L \equiv -\dfrac{P^{(x+1)/2}}{Q^{(y-1)/2}} \equiv (-1)^{(x-1)/2} \equiv \pm 1 \pmod K.
\end{equation}
Hence, by \eqref{3.12} and \eqref{3.14}, the solution \eqref{3.13} belongs to $S(1)$. Recall that $(1,1,2)$ is the least solution of $S(1)$. Applying Lemma \ref{Lemma 2.1} to \eqref{3.13}, we have
\begin{equation}\label{3.15}
z=2t, \ t \in \mathbb{N}, \ 2\nmid t, \ t>1,
\end{equation}
\begin{equation}\label{3.16}
P^{(x-1)/2}\sqrt{P}+Q^{(y-1)/2}\sqrt{-Q} = \lambda_1(\sqrt{P}+\lambda_2\sqrt{-Q})^t, \ \lambda_1,\lambda_2 \in \{1,-1\}.
\end{equation}

By \eqref{3.16}, we get
\begin{equation}\label{3.17}
P^{(x-1)/2}= \lambda_1 \sum_{i=0}^{(t-1)/2} \binom {t}{2i} P^{(t-1)/2-i}(-Q)^i, \ \ 
Q^{(y-1)/2}= \lambda_1 \lambda_2 \sum_{i=0}^{(t-1)/2} \binom {t}{2i+1} P^{(t-1)/2-i}(-Q)^i.
\end{equation}
Since $(x,y,z)\neq(1,1,2)$, we have $\max \{x,y\}>1$. Hence, by \eqref{3.1} and \eqref{3.17}, we get either $P\mid t$ or $Q\mid t$. This implies that
\begin{equation}\label{3.18}
t\geq \min \{P,Q\}.
\end{equation}
Further, since $\min \{P,Q\}>30$, we see from \eqref{3.18} that
\begin{equation}\label{3.19}
t>30.
\end{equation}

Let
\begin{equation}\label{3.20}
\alpha=\sqrt{P}+\sqrt{-Q}, \ \beta=\sqrt{P}-\sqrt{-Q}.
\end{equation}
Then $(\alpha+\beta)^2=4P,\ \alpha\beta=K^2$ are nonzero coprime integers, and $\alpha/\beta=((P-Q)+2\sqrt{-PQ})/K^2$ is not a root of unity. Hence, $(\alpha,\beta)$ is a Lehmer pair. By \eqref{2.5}, \eqref{3.16}, \eqref{3.17} and \eqref{3.20}, we have
\begin{equation}\label{3.21}
Q^{(y-1)/2}= \bigg|\dfrac{\alpha^t-\beta^t}{\alpha-\beta}\bigg|=|L_t(\alpha,\beta)|.
\end{equation}
We see from \eqref{3.21} that the Lehmer number $L_t(\alpha,\beta)$ has no primitive divisors. But, by Lemma \ref{Lemma 2.2}, we find from \eqref{3.19} that this is false. Thus, under the assumption, \eqref{1.1} has only the solution $(x,y,z)=(1,1,2)$. The theorem is proved.

\end{proof}

\bigskip
\subsection*{Acknowledgements}
The first and third authors were supported by T\"{U}B\.{I}TAK (the Scientific and Technological Research Council of Turkey) under Project No: 117F287.

\end{document}